\documentclass{daj}

\usepackage{amsmath,amsthm,amssymb,fullpage,enumerate,cleveref,url,titlesec}

\newtheorem{theorem}{Theorem}[section]

\newtheorem{lemma}[theorem]{Lemma}
\newtheorem{definition}[theorem]{Definition}
\newtheorem{corollary}[theorem]{Corollary}

\newtheorem{conjecture}[theorem]{Conjecture}

\usepackage{todonotes}
\usepackage{xcolor}
\usepackage{relsize}
\usepackage{comment}
\usetikzlibrary{positioning}

\newcommand{\F}{\mathcal{F}}

\DeclareMathAccent{\widehat}{\mathord}{largesymbols}{"62}

\dajAUTHORdetails{%
  title = {Monochromatic Sums and Products of Polynomials}, 
  author = {Ryan Alweiss},
  plaintextauthor = {Ryan Alweiss},
  plaintexttitle = {Monochromatic Sums and Products of Polynomials}, 
  runningtitle = {Monochromatic Sums and Products of Polynomials}, 
  runningauthor = {Ryan Alweiss},
  copyrightauthor = {Ryan Alweiss},
  keywords = {hindman, partition regular},
}   

\dajEDITORdetails{%
   year={2024},
   number={5},
   received={17 February 2023},   
   published={27 August 2024},  
   doi={10.19086/da.117575},       
}   

\begin{document}

\begin{frontmatter}[classification=text]

\title{Monochromatic Sums and Products of Polynomials} 

\author[ra]{Ryan Alweiss\thanks{Research supported by an NSF Mathematical Sciences Postdoctoral Research Fellowship.}}

\begin{abstract}
We show that the pattern $\{x,x+y,xy\}$ is partition regular over the space of formal integer polynomials of degree at least one with zero constant term, with primitive recursive bounds.  This provides a new proof for the partition regularity of $\{x,x+y,xy\}$ over $\mathbb{N}$, which gives the first primitive recursive bound.
\end{abstract}
\end{frontmatter}


\section{Introduction}~\label{section:introduction}

A pattern $\F$ is \emph{partition regular} over a set $S$, usually taken to be $\mathbb{N}$, if in any coloring of $S$ with finitely many colors, there is some monochromatic copy of $\F$.  Partition regularity is an important subfield of Ramsey theory, and encompasses results such as van der Waerden's theorem, which states that the pattern $\{x,x+d,\cdots, x+nd\}$ is partition regular over $\mathbb{N}$, and Schur's theorem, which states that $\{y,x,x+y\}$ is partition regular over $\mathbb{N}$. In 1933, Rado~\cite{rado} completely characterized partition regularity over $\mathbb{N}$ when $\F$ is taken to be the solution set to a linear system of equations.

Over polynomials, the situation is much less well understood.  The pattern $\{x,y,xy\}$ is known to be partition regular over $\mathbb{N}$, because we can color powers of $2$ and then find a pattern $\{2^x,2^y,2^{x+y}\}$ using Schur. It is natural to wonder whether there is a common generalization of this multiplicative form of Schur's theorem and the ordinary additive form.  This question, of whether $\{x,y,x+y,xy\}$ is partition regular over $\mathbb{N}$, has been repeated by several authors (see e.g. \cite{bergelson}, \cite{hindman}, \cite{moreira}, \cite{green}), and is the first unsolved case of a central and longstanding open problem in partition regularity due to Hindman \cite{hindold}. In a 2016 breakthrough, Moreira \cite{moreira} proved that the pattern $\{x,x+y,xy\}$ is partition regular over $\mathbb{N}$.  Before 2016, it was not even known if $\{x+y,xy\}$ was partition regular. Moreira's proof uses extensively the following notion from dynamics.

\begin{definition}
	A subset $S$ of $\mathbb{N}$ is \emph{syndetic} if it has bounded gaps, i.e. there exists a finite set $F \subset \mathbb{Z}$ so that $S+F \supset \mathbb{N}$.  A subset $T$ of $\mathbb{N}$ is \emph{thick} if it contains arbitrary long intervals.  A \emph{piecewise syndetic} subset $U$ is the intersection of a thick set with a syndetic set.
	\end{definition} 
	
Moreira iterates the ``piecewise syndetic" version of the van der Waerden theorem, in which the size of the set $F$ blows up very quickly. Thus, he does not obtain primitive recursive bounds.

Following Moreira's work, there has recently been exciting progress on Hindman's conjecture, with Bowen and Sabok \cite{bowenrational} proving that $\{x,y,x+y,xy\}$ is partition regular over the rationals. Over the integers, the problem is much more difficult.  Recently, Bowen \cite{bowen} proved that any two-coloring of $\mathbb{N}$ has infinitely many monochromatic $\{x,y,x+y,xy\}$.

In this paper, we prove that $\{x,x+y,xy\}$ is partition regular in a more general setting. Let a \emph{good} polynomial be a polynomial with zero constant term and non-negative integer coefficients.  Let $T$ be the semi-ring of nonzero good polynomials of a variable $t$.  In other words, $T$ consists of all polynomials of $t$ with nonnegative integer coefficients and $0$ constant term.  It is important that $T$ be defined with no constant term, as otherwise since $\mathbb{N} \subset T$ this would not be any stronger than Hindman over $\mathbb{N}$. 

We can add and multiply freely in $T$, and thus consider polynomial expressions of $x,y \in T$.  As such we can consider the partition regularity of families of polynomial equations, where our variables $x,y$ are in $T$ and are themselves polynomials of $t$.  For instance, the pattern $\{x,y,xy\}$ is partition regular in $T$, because we can always find some $x=t^a$ and $y=t^b$ so that $t^a, t^b, t^{a+b}$ are all the same color. 

\begin{theorem}
	For any set of good polynomials $P_1, \cdots, P_k$, the pattern $\{x,x+P_1(y), \cdots, x+P_k(y), xy\}$ is partition regular over $T$ with primitive recursive bounds, and is thus also partition regular over $\mathbb{N}$ with primitive recursive bounds.
\end{theorem}

In particular, we have the following corollary.

\begin{corollary}
	The pattern $\{x+y,xy\}$ is partition regular over $\mathbb{N}$ with primitive recursive bounds, i.e. there is some primitive recursive $N=N(n)$ so that when $[N]=\{1,2, \cdots, N\}$ is colored with $n$ colors, there is a monochromatic $\{x+y,xy\}$.
\end{corollary}

Moreira's proof \cite{moreira} does not obtain any explicit bounds.  Furthermore, his proof works for polynomials over a finite field, but not for polynomials with rational or integer coefficients; he notes specifically that his Lemma 7.4 breaks down.  Partition regularity over $T$ (with primitive recursive bounds) is stronger than partition regularity over $\mathbb{N}$ (with primitive recursive bounds); we prove this in detail in Section~\ref{section:discussion}.  

An important feature of our proof is that it considers high degree polynomials, which we believe is necessary for settling Hindman's conjecture.  Previous work in this area implicitly works over a setting in which we conjecture that Hindman's conjecture is actually false.  This is discussed more in Section~\ref{section:discussion}. Because our proof is the first to take advantage of higher degree polynomials, it relies in an intrinsic way on the polynomial van der Waerden theorem.

Whereas Moreira uses the linear van der Waerden theorem to prove $\{x,x+y,xy\}$ is partition regular, we prove a better bound by instead using the polynomial van der Waerden theorem.  Recall the polynomial van der Waerden theorem of Bergelson and Leibman \cite{bl2}, which states that for any set of good polynomials $P_1, \cdots, P_k$, the pattern $x+P_1(y), \cdots, x+P_k(y)$ is partition regular over $\mathbb{N}$.  This holds over any semigroup \cite{moreira} and there is a simple color-focusing proof \cite{walters} straightforwardly generalizing the ordinary color-focusing proof of van der Waerden theorem.  A result of Shelah \cite{shelah} provides primitive recursive bounds in the polynomial Hales-Jewett theorem of Bergelson and Leibman \cite{polyhj}, and thus over any semigroup \cite{moreira}.  

To some extent, working over the polynomials themselves can be viewed as analogous to extending results like van der Waerden theorem to the Hales-Jewett theorem; working over these more general (semi)-rings like $T$ abstracts away the properties of $\mathbb{N}$, and captures the combinatorial structure.  In particular, we avoid tools from dynamics such as the piecewise syndetic version of the van der Waerden theorem, and in fact our proof is ``purely combinatorial" does not use the notions of syndetic, piecewise syndetic, or thickness at all.  This again distinguishes our proof from the existing proofs in the literature. 

We hope our new perspective will continue to find further applications and shine some light on Hindman's conjecture. Using the methods developed in this paper, the author subsequently settled the original conjecture of Hindman over the rationals \cite{alweiss}. 

The structure of the paper is as follows.  First, in Section \ref{section:3color}, we describe a proof of the partition regularity of the pattern $\{x,x+y,xy\}$ in the $3$-color case, to convey the general idea of our paper. Then, we discuss our proof of $\{x,x+y,xy\}$ for arbitrarily many colors in Section \ref{section:general}. Finally we conclude in Section \ref{section:discussion} with some discussions and open problems. 

\section{The $3$-color case}~\label{section:3color}

We will first describe the proof of the partition regularity of the pattern $\{a,a+b,ab\}$ for polynomials over $3$ colors, red, blue, and green.

We will call a polynomial $P$ of one or more variables \emph{good} if it has zero constant term and nonnegative integer coefficients.  If $x$ is an element of some semi-ring or ring $R$ and $P$ is a good polynomial, then if $x$ is in $R$, $P(x)$ will also be in $R$, since rings and semi-rings are closed under addition and multiplication.  Let $T$ be the semi-ring of nonzero good polynomials of a single variable $t$; for us, a semi-ring or ring does not necessarily have $1$.  Our goal is to prove that $\{x,x+y,xy\}$ is partition regular over $T$.  In other words, we will prove that in any finite coloring of $T$, there exist some $x=x(t)$ and $y=y(t)$ which are good polynomials of $T$ so that $x,x+y,xy$ are all the same color.  The \emph{size} $s(P)$ of a good polynomial $P$ is the largest $c$ so that its degree and its coefficients are all at most $c$. The size of the $0$ polynomial will be $0$. Recall the polynomial van der Waerden theorem \cite{bl2}, which by the polynomial Hales-Jewett theorem holds over any semigroup \cite{joelblog}.

Assume that there exists a coloring of the semi-ring $T$ without the pattern $\mathcal{F}=\{a,a+b,ab\}$.  By the polynomial van der Waerden theorem, for any $c$, there exist some $x,d$ so that $x+P(d)$ is red for all good $P$ of size at most $c$.  Pick $c=c_0$ to be specified later.

Now, because we do not have the pattern $\{a,a+b,ab\}$, and for any good $P$ of size less than $c_0$, and we have the pattern $\{(x+P(d)),(x+P(d))+d\}$, we know $d(x+P(d))$ cannot be red for any good $P$ of size less than $c_0$. Hence, the terms $dx+dP(d)$ are all blue or green.  

Again by the polynomial van der Waerden theorem over semigroups \cite{joelblog}, if $c_0$ is large enough, we can find without loss of generality some $f,e$ of the form $dP(d)$ (even of the form $P(d^2)$) so that $dx+f+Q(e)$ are all blue for good $Q$ of size at most $c_1$ and that $x+\frac{f}{d}+P(d,\frac{e}{d})$ are red for good $P$ of size at most $c_1$, as we will have $x+\frac{f}{d}+P(d,\frac{e}{d})=x+P'(d)$ for good $P'$ of size at most $c_0$.  We rename $x+\frac{f}{d}$ as $x$, so that $x+P(d,\frac{e}{d})$ are all red, while now $dx+P(e)$ are all blue, again for all good $P$ of size at most $c_1$.  In particular, $c_0$ will be a primitive recursive function of $c_1$.

Now, because $\{x,x+de\}$ and $\{x+\frac{e}{d},x+\frac{e}{d}+de\}$ are red, $dex$ and $dex+e^2$ are not red.  As $\{dx,dx+e\}$ and $\{dx+e,dx+e+e\}$ are blue, they also cannot be blue.  So $dex, e(dx+e)=dex+e^2$ are both green.  This means $de^3x$ cannot be green.  Because $\{x,x+de^3\}$ are red, $de^3x$ cannot be red.  But because $\{dx,dx+e^3\}$ are blue, $de^3x$ cannot be blue.  This gives a contradiction, and we are done.

\section{Partition Regularity of Sums and Products in $T$}~\label{section:general}

We now present the proof of the general case.  The idea is the same as in the last section, although it is a little more technical to write out formally.  We prove the following lemma.  The $c_i$ will shrink sufficiently quickly; what matters for us is that $c_i$ is a primitive recursive function of $c_{i+1}$.

\begin{lemma}
	
For any coloring $C$ of $T$ and any positive integers $k$ and $c_k$, there exist some sequence of positive integers $c_0, c_1, c_2, \cdots, c_k$ and elements $b_1, \cdots, b_k$ of $T$, so that the following holds.

Let $a_1=b_1$ and $a_i=a_{i-1}b_i$ for $2 \le i \le k$.  In particular, $b_i=\frac{a_i}{a_{i-1}}$ for $i \ge 2$.  

Then there exists some $x_k$ so that for all good polynomials $P$ of size at most $c_i$ and all $0 \le i<k$, $$\left(\prod_{j=1}^{i}a_j\right)x_k+P(a_{i+1},b_{i+2}, \cdots, b_k)$$ is the same color in the coloring $C$ as $\left(\prod_{j=1}^{i}a_j\right)x_k$.  

\end{lemma}

Note that when $i=0$, $\left(\prod_{j=1}^{i}a_j\right)x_0=x_0$, and when $i=k-1$, $P(a_{i+1}, \cdots)=P(a_k)$.  

\begin{proof}

We induct on $k$.  The base case $k=0$ is trivial.  Assume now that the lemma holds for some $k$.  We assume that $c_k$ is a large enough function of $c_{k+1}$, to be specified later.  By the polynomial van der Waerden theorem over semigroups \cite{joelblog}, for any $n$ there exists some $f,e$ which are good polynomials of $a_k^{k+1}$ so that for any good polynomial $P$ with $s(P) \le c_{k+1}$, all of $\left(\prod_{j=1}^{k}a_j\right)x_k+f+P(e)$ are the same color as $\left(\prod_{j=1}^{k}a_j\right)x_k+f$. 

Now, we set $x_{k+1}=x_k+\frac{f}{\prod_{j=1}^{k}a_i}$ and $b_{k+1}=\frac{e}{a_k}$ so that $e=a_{k+1}$. For $0 \le i<k$ and $P$ of size at most $c_{k+1}$, $$\left(a_1 \cdots a_i\right)(x_{k+1})+P(a_{j+1}, b_{j+2}, \cdots b_k, b_{k+1})=\left(a_1 \cdots a_i\right)\left(x_k+\frac{f}{\prod_{j=1}^{k}a_j}\right)+P(a_{i+1}, b_{i+2}, \cdots, b_k, b_{k+1})$$ $$=\left(a_1 \cdots a_i\right)x_k+\frac{f}{a_{i+1} \cdots a_k}+P(a_{i+1},b_{i+2}, \cdots b_k, b_{k+1})$$ $$=(a_1 \cdots a_j)x_k+P'(\frac{a_k}{a_{i+1}}, \cdots, \frac{a_k}{a_{k-1}},a_k)+P(a_{i+1},b_{i+2}, \cdots b_k, b_{k+1})$$ $$=(a_1 \cdots a_i)x_k+P''(a_{i+1}, b_{i+2}, \cdots, b_k, b_{k+1})$$ for some good polynomial $P''$, recalling again that $f$ is a polynomial of $a_k^{k+1}$.  We can ensure that $P''$ has size less than $c_k$, by choosing $c_k$ large enough.

For $i=k$, $$\left(a_1 \cdots a_k\right)(x_{k+1})+P(a_{k+1})=\left(a_1 \cdots a_k\right)\left(x+\frac{f}{\prod_{j=1}^{k}a_j}\right)+P(e)=(a_1 \cdots a_k)x+f+P(e)$$ has the same color as $\left(a_1 \cdots a_k\right)x_{k+1}=(a_1 \cdots a_k)x_{k+1}+f$ for all $s(P) \le c_{k+1}$. \end{proof}

%
%

This proof yields primitive recursive bounds.  In particular, $c_k$ will be a primitive recursive function of $c_{k+1}$, as due to Shelah \cite{shelah}, we have primitive recursive bounds in the polynomial Hales-Jewett theorem and thus in polynomial van der Waerden over a semigroup.

By passing to a subset of the $(\prod_{j=1}^{i} a_j)x$ with a fixed color, we can finish.  Explicitly, if assume without loss of generality that $a_1 \cdots a_{\ell_i}$, and so on are the same color for $\ell_1<\ell_2<\cdots$, and set $c_1=a_1 \cdots a_{\ell_1}$, $c_2=a_{\ell+1} \cdots a_{\ell_2}$, and in general $c_n=a_{\ell_{n-1}+1} \cdots a_{\ell_n}$, then all $\left(\prod_{j=1}^i c_j\right)x+P(c_{j+1}, \cdots, c_n)$ are the same color.  In particular, $\{x,x+P_1(c_1), \cdots, x+P_k(c_1),c_1x\}$ are the same color, for any finite family of good polynomials $P_i$.  What we actually get is slightly more general, and very similar to Theorem 1.4 of \cite{moreira}.

\section{Discussions and Open Problems}~\label{section:discussion}

Some discussion about partition regularity over various rings and semi-rings is in order.  There is an important way in which our proof differs from existing proofs in the literature. For $d \ge 1$, let $P_d$ be the set of polynomials of a countable set of variables $t=x_1, x_2, \cdots, x_n, \cdots$ with nonnegative integer coefficients, constant term, and degree at most $d$ in each variable.  Let $P=P_{\infty}$.  In particular, $P \supset T$.  Note that for $d<\infty$, $P_d$ is actually not a semi-ring, as for instance it contains $x_1$ but does not contain $x_1^{d+1}$.

Moreira's result \cite{moreira} is stated over $\mathbb{N}$, and unlike our proof, it is more natural in that setting.  However, as it only uses multilinear polynomials, it can be made to work over $P_1$ \cite{bowen} using the methods of \cite{bowencomposite}.  Previous work in the area, such as Bowen and Sabok's proof of the partition regularity of $\{x,y,x+y,xy\}$ over the rationals \cite{bowenrational} and Bowen's proof of the $2$ color case \cite{bowen2color} can similarly be made to work over, say, $P_5$.  In other words, these proofs are not using high degree polynomials.  We conjecture this is insufficient to settle the whole Hindman conjecture.

\begin{conjecture}
	If $1 \le d<\infty$, then $\{x,y,x+y,xy\}$ is not partition regular over $P_d$.
\end{conjecture}

We do know that this conjecture is true for $d=1$. It is a nice exercise to check that the $2$-coloring of $P_1$ where reducible polynomials are red and irreducible polynomials are blue avoids a monochromatic $\{x,y,x+y,xy\}$. We make a complementary conjecture.

\begin{conjecture}
	
For all $1 \le i<j$, there exists some pattern $\F$ so that $\F$ is partition regular in $P_j$ but not in $P_i$, but $\F$ does appear in $P_i$ (i.e. the $1$-coloring of $\F$ contains a monochromatic $P_i$).

\end{conjecture}

In particular, the previous conjecture says that we do \emph{not} think the pattern $\{x,y,x+y,xy\}$ is an example of such a pattern.  Clearly if $1 \le i \le j \le \infty$, and $F$ is partition regular over $P_i$, then it is partition regular over $P_j \supset P_i$.

\begin{lemma}~\label{lemma:contain}
A $\F$ is partition regular over $P=P_{\infty}$ if and only if it is partition regular over $T$.  If a pattern is partition regular over $T$ (with primitive recursive bounds), then it is partition regular over~$\mathbb{N}$ (with primitive recursive bounds).	 
\end{lemma}

\begin{proof}
Clearly if some pattern is partition regular over $T$, then it is partition regular over $P \supset T$. If a pattern $\F$ is not partition regular over $T$, then we can find a coloring $c$ of $T$ avoiding $\F$, and color any $P(x_1, \cdots, x_n) \in P$ with the color of $P(t,t, \cdots, t)$ in $c$, and so $\F$ will not be partition regular over $P$ either.

If there is a coloring $c$ of $\mathbb{N}$ avoiding $\F$, then we can color $p(t) \in T$ with the color of $p(2)$ in $c$.  This coloring of $T$ will avoid $\F$.  Thus, partition regularity over $T$ (with primitive recursive bounds) in turn implies partition regularity over $\mathbb{N}$ (with primitive recursive bounds). \end{proof}

	
Any nontrivial insight into the relationship between partition regularity over various settings would be interesting.  We conjecture that the second part of Lemma~\ref{lemma:contain} is not strict.

\begin{conjecture}
	A polynomial pattern $\F$ is partition regular over $T$ if and only if it is partition regular over~$\mathbb{N}$.  
\end{conjecture}

In principle, one could ask about the problem for other semi-rings. Recall that an integral domain $D$ has \emph{characteristic $0$} if for any positive integer $k$, there exist some $x \in D$ so that $kx:=x+ \cdots +x \neq 0$.  The semi-ring $T$ is the most important semi-ring to consider in the following sense.


\begin{lemma}
	If some family $\mathcal{F}$ of good polynomials $P_1(a_1, \cdots, a_m), \cdots, P_k(a_1, \cdots, a_m)$ is partition regular over the semi-ring $T$, then it is partition regular over any integral domain $R$ of characteristic $0$.
\end{lemma}

\begin{proof}
Assume $\mathcal{F}$ is not partition regular over $R$.

Say that there exists a sequence $t_1, \cdots, t_n, \cdots$ in $R$ so that $P(t_n) \neq 0$ for all good polynomials $P$ of size at most $n$.  By a standard compactness argument, there is a limit $t$, and a coloring of all good polynomials of $t$ avoiding $\mathcal{F}$.

If there exists a $c$ so that for all $x$, there is some good polynomial $P$ of size at most $c$, so that $P(x)=0$, then taking the product of all of these polynomials, there exists some $P$ so that $P(x)=0$ identically for $x \in R$.  If $P$ is not a monomial and has degree $d$, then we may pick an appropriate positive integer $k$ so that $Q(x)=P(kx)-k^dP(x)$ is also identically $0$, but is a nonzero polynomial with integer coefficients of smaller degree.  $Q$ may not be good, but this is not a problem, all we require is that it has integer coefficients.  We can reduce the degree until we find some $kx^d$ which is identically $0$ for $x \in R$.  Because $R$ is an integral domain, this means that $kx=0$, a contradiction. \end{proof}

Partition regularity over $T$ is also interesting for the following reason.  If there is a coloring of $T$ that avoids some polynomial pattern $\F$, we can assume without loss of generality that it is \emph{polynomially syndetic}, i.e. there exists a $c$ so that for any element $x$ and any color $K$, there is some polynomial $P$ of size at most $c$ so that $P(x)$ is color $K$.  This is again by compactness; if there exists a color $K$ and a sequence $x_1, \cdots, x_n, \cdots$ so that $P(x_n)$ is never color $K$ when $P$ is size at most $c$, then we can create a coloring of good polynomials of a limit $x$ that avoids $\F$ and does not use color $K$, and reduce the number of colors used.  





\section*{Acknowledgments} 
Thanks to Matthew Bowen and Marcin Sabok for helpful mathematical conversations, and Noga Alon, Matija Bucic, Sayan Goswami, Timothy Gowers, Zach Hunter, Noah Kravitz, Imre Leader, Joel Moreira, Ashwin Sah, and an anonymous reviewer for helpful comments on the exposition.

\bibliographystyle{amsplain}


\begin{dajauthors}
\begin{authorinfo}[ra]
  Ryan Alweiss\\
  University of Cambridge\\
  Cambridge, UK\\
  ra699\imageat{}cam\imagedot{}ac\imagedot{}uk \\
  \url{https://sites.google.com/view/ryan-alweiss/home}
\end{authorinfo}
\end{dajauthors}

\end{document}